\documentclass[hidelinks,onefignum,onetabnum]{siamart250211}

\usepackage{subcaption, booktabs, multicol, multirow, makecell}

\newcommand{\dd}{\,{\rm d}}
\newcommand{\dual}[1]{\left\langle {#1} \right\rangle}
\newtheorem{example}[theorem]{Example}

\newtheorem{assumption}{Assumption}

\usepackage{amsmath}

\usepackage{tikz}
\usetikzlibrary{positioning}

\newcommand{\fb}{f_{-\mu}}

\usepackage{enumitem}
\usepackage{comment}

\definecolor{atomictangerine}{rgb}{1.0, 0.6, 0.4}

\usepackage{lipsum}
\usepackage{amsfonts}
\usepackage{graphicx}
\usepackage{epstopdf}
\usepackage{algorithmic}
\ifpdf
  \DeclareGraphicsExtensions{.eps,.pdf,.png,.jpg}
\else
  \DeclareGraphicsExtensions{.eps}
\fi

\newsiamremark{remark}{Remark}
\newsiamremark{hypothesis}{Hypothesis}
\crefname{hypothesis}{Hypothesis}{Hypotheses}
\newsiamthm{claim}{Claim}
\newsiamremark{fact}{Fact}
\crefname{fact}{Fact}{Facts}

\headers{Accelerated Mirror Descent Methods}{Accelerated Mirror Descent Methods}

\title{Accelerated Mirror Descent Method through Variable and Operator Splitting\thanks{Submitted to the editors DATE.
\funding{}}}

\author{
Long Chen \thanks{Department of Mathematics, University of California, Irvine, CA 92697, USA. 
 (\email{chenlong@math.uci.edu}).}
\and Hao Luo \thanks{National Center for Applied Mathematics in Chongqing, Chongqing Normal University, Chongqing, 401331, China; Chongqing Research Institute of 
 Big Data, Peking University, Chongqing, 401121, China (\email{luohao@cqnu.edu.cn})}
\and Jingrong Wei \thanks{Department of Mathematics, The Chinese University of Hong Kong, Shatin, NT, Hong Kong SAR, China. (\email{jingrongwei@cuhk.edu.hk}).} 
\and Zeyi Xu \thanks{Department of Mathematics, University of California, Irvine, CA 92697, USA. 
 (\email{zeyix1@uci.edu}).}
 \and Yuan Yao \thanks{Department of Mathematics, Hong Kong University of Science and Technology, Clear Water Bay, Kowloon, Hong Kong SAR, China (yuany@ust.hk)}
 }

\usepackage{amsopn}

\ifpdf
\hypersetup{
  pdftitle={Accelerated Mirror Descent Method through Variable and Operator Splitting},
  pdfauthor={Long Chen, Hao Luo, Jingrong Wei, Zeyi Xu and Yuan Yao}
}
\fi

\begin{document}

\maketitle

\begin{abstract}
Mirror descent uses the mirror function to encode geometry and constraints, improving convergence while preserving feasibility. Accelerated Mirror Descent Methods (Acc-MD) are derived from a discretization of an accelerated mirror ODE system using a variable--operator splitting framework. A geometric assumption, termed the Generalized Cauchy-Schwarz (GCS) condition, is introduced to quantify the compatibility between the objective and the mirror geometry, under which the first accelerated linear convergence for Acc-MD on a broad class of problems is established. Numerical experiments on smooth and composite optimization tasks demonstrate that Acc-MD exhibits superior stability and competitive acceleration compared to existing variants.
\end{abstract}



\section{Introduction}
Consider the unconstrained convex optimization problem
$$
\min_{x \in \mathbb{R}^n} f(x),
$$
where $f$ is convex and differentiable. Mirror descent~\cite{nemirovskij1983problem} updates as
\begin{equation}\label{eq:MD}
    x_{k+1} = \arg\min_{x \in \mathbb{R}^n} \left\{
    \langle \nabla f(x_k), x \rangle + \frac{1}{\alpha_k} D_\phi(x, x_k)
    \right\},
\end{equation}
where $\alpha_k > 0$ is the step size, $\phi$ is a smooth, strictly convex {\it mirror function} of Legendre type, and $D_\phi(\cdot, \cdot)$ is the associated Bregman divergence. In \cite{nesterov2005smooth}, Nesterov proposed an accelerated version that later inspired the mirror--prox method \cite{nemirovski2004prox}. 

The mirror function encodes both geometry and constraints, with applications in imaging, phase transition, signal processing, and machine learning~\cite{Osher2016,ben2001ordered,Yao23,liang2024differentially,bao2024convergence,lan2023policy}.
One use of mirror descent is as a nonlinear preconditioner. Relative smoothness is defined by $D_f(x,y)\le L D_\phi(x,y)$. The Lipschitz constant of $\nabla f$ is $L_f$. Mirror descent improves convergence when $L\ll L_f$, especially if $L_f=\infty$ (i.e., when $f$ is not smooth in the Euclidean sense).

Another use of mirror descent is to impose constraints through the mirror function. For example, the entropic mirror descent method uses Shannon’s entropy $\phi(x)=\sum_i x_i\log x_i$, which ensures that $x_i$ stay in the unit simplex. In~\cite{beck_mirror_2003}, mirror descent is seen as a nonlinear projected subgradient-type method for constrained optimization.

While acceleration is well-understood in Euclidean spaces, extending it to general Bregman geometries remains challenging due to the lack of a natural coupling between the primal and dual metrics. Recently, acceleration has been studied by viewing optimization algorithms as discretizations of continuous-time ODE dynamics. Krichene, Bayen, and Bartlett~\cite{krichene2015accelerated} extended the ODE framework for Nesterov acceleration~\cite{Su;Boyd;Candes:2016differential} to accelerated mirror descent, showing that suitable discretizations yield a $\mathcal{O}(L_f/k^2)$ rate. Wibisono, Wilson, and Jordan~\cite{wibisono2016variational} derived a Bregman Lagrangian that captures accelerated flows. Yuan and Zhang~\cite{yuan2023analysis} developed high-resolution ODEs and recovered the optimal $\mathcal{O}(L_f/k^2)$ rate for accelerated mirror descent methods in~\cite{nesterov2005smooth}. Also starting from Nesterov acceleration, Yu et al.~\cite{yu2025mirrordescentgeneralizedsmoothness} extended the method to work under generalized smoothness~\cite{li2023convex} and established the $O(1/k^2)$ rate. Other variants are in~\cite{Allen-Zhu2016,d2021acceleration}.

However, these analyses~\cite{krichene2015accelerated, nesterov2005smooth,wibisono2016variational,yu2025mirrordescentgeneralizedsmoothness,yuan2023analysis,Allen-Zhu2016,d2021acceleration} do not use relative smoothness, $D_f(x,y) \le L D_\phi(x,y)$. Thus, convergence guarantees can be loose outside the Euclidean setting, especially when $L$ is small but $L_f$ is large or infinite. Birnbaum, Devanur, and Xiao~\cite{BirnbaDevanuXiao2011Distributed} showed that mirror descent achieves a $\mathcal{O}(L/k)$ rate under relative smoothness. For composite problems $F(x)=f(x)+g(x)$, where $g$ is convex but possibly non-smooth, the Bregman Proximal Gradient (BPG) method~\cite{Teboul2018simplified} achieves a $\mathcal{O}(L/k)$ rate; see also~\cite{bauschke2017descent,lu2018relatively,ZhouLiangShen2019simple}. In~\cite{lu2018relatively}, Lu, Freund, and Nesterov introduced relative strong convexity, $\mu D_\phi(x,y) \le D_f(x,y)$, and proved a linear convergence rate of $1-\mu/L$ for mirror descent. Later, Hanzely and  Richt{\'a}rik \cite{hanzely2021fastest} extended the linear convergence to stochastic mirror descent .

Without extra structure, mirror descent methods are limited to $\mathcal{O}(1/k)$ complexity~\cite{DragomTaylordAspreBolte2022Optimal}. Extra conditions are needed for acceleration under relative smoothness. Under the gradient dominated condition and  the $\theta$-uniform condition, the Bregman gradient method achieved linear convergence  \cite{bauschke2019linear}. The Accelerated Bregman Proximal Gradient (ABPG) method~\cite{Hanzely2021} obtains a rate $\mathcal{O}(L/k^{\gamma})$ based on the triangle scaling exponent (TSE) $\gamma \le 2$. Full acceleration $\mathcal{O}(L/k^2)$ requires $\gamma=2$. Notably, no accelerated linear rate is known under relative strong convexity.

Table~\ref{tab: MD conv result} summarizes recent convergence results for mirror descent and accelerated mirror descent methods.

\begin{table}[htp]
\centering
\caption{Convergence rates of mirror descent and accelerated mirror descent methods. 
The mirror function $\phi$ is assumed strongly convex and differentiable of Legendre type.
}
\resizebox{\linewidth}{!}{
\renewcommand{\arraystretch}{1.275}
\begin{tabular}{@{} c  c  c @{}}
\toprule
Algorithm / Theory 
& Assumptions 
& Convergence rate \\ 
\midrule

\makecell{Nesterov (2005)~\cite{nesterov2005smooth} \\
Nemirovski (2005)~\cite{nemirovski2004prox}}
& $f$: $L_f$-smooth
& $\mathcal{O}(1/k)$ \\
\midrule

\makecell{Krichene et. al. (2015)~\cite{krichene2015accelerated} \\
Yuan and Zhang (2024)~\cite{yuan2023analysis}}
& $f$: $L_f$-smooth
& $\mathcal{O}(L_f/k^2)$ \\
\midrule

Birnbaum et. al. (2011)~\cite{BirnbaDevanuXiao2011Distributed}
& $f$: $L$-relative smooth
& $\mathcal{O}(L/k)$ \\
\midrule


Hanzely et. al. (2021)~\cite{Hanzely2021}
& \makecell{triangle scaling exponent $\gamma \le 2$ \\[1pt]
$f$: $L$-relative smooth}
& $\mathcal{O}(L/k^\gamma)$ \\
\midrule

\makecell{Lu et. al. (2018)~\cite{lu2018relatively} \\
Teboulle (2018)~\cite{Teboul2018simplified}}
& \makecell{$f$: $\mu$-relative convex, $L$-relative smooth}
& $\mathcal{O}\!\big((1-\mu/L)^k\big)$ \\
\midrule

\makecell{(\textbf{New}) Algorithm~\ref{alg:AccMD}, \ref{alg:AccMDback} \\ Theorem~\ref{thm:convergence rate of Acc-MD dual}}
& \makecell{$f$: $\mu$-relative convex, $L$-relative smooth \\ generalized Cauchy-Schwarz \eqref{eq: A2}}
& \makecell{$\mathcal{O}\!\big((1+\sqrt{\mu/C_{f,\phi}})^{-k}\big)$} \\
\midrule

\makecell{(\textbf{New}) Algorithm~\ref{alg: restart perturbed VOS} \\ Theorem~\ref{thm:convergence rate of perturbed Acc-MD dual}}
& \makecell{$f$: $L$-relative smooth,\\
 generalized Cauchy-Schwarz \eqref{eq: A2}}
& $\mathcal{O}(C_{f,\phi}/k^2)$ \\
\bottomrule
\end{tabular}}
\label{tab: MD conv result}
\end{table}

In this work, we develop accelerated mirror descent methods. Using the variable and operator splitting (VOS) framework from~\cite{chen2025accelerated}, we propose a new accelerated mirror descent (Acc-MD) flow:
\begin{equation}\label{eq:VOSflow}  
\left\{
\begin{aligned}
    x^{\prime} &= y - x, \\
    \left(\nabla \phi(y)\right)^{\prime} &= - \mu^{-1} \nabla f(x) + \nabla \phi(x) - \nabla \phi(y),
\end{aligned}
\right.
\end{equation}
with $x(0)=x_0$ and $y(0)=y_0$. The variable is split into $x$ and $y$. At equilibrium, $y^{\star}=x^{\star}$ and $\nabla f(x^{\star})=0$. The variable $y$ allows operator splitting $\nabla f(x)$ and $\nabla \phi(y)$.

From discretizations of~\eqref{eq:VOSflow}, we design Acc-MD methods: \cref{alg:AccMD} and \cref{alg:AccMDback}. We introduce a new assumption \eqref{eq: A2}, a generalized Cauchy-Schwarz (GCS) inequality with constant $C_{f,\phi}$. Based on GCS, We establish the first accelerated linear convergence result for mirror descent, with rate $(1+\sqrt{\mu/C_{f,\phi}})^{-1}$. We show cases where GCS \eqref{eq: A2} holds and $C_{f,\phi}$ is explicit. 

As seen in Figure \ref{fig:diagram} (Section \ref{sec:A2}), GCS covers important geometries, such as Example \ref{ex: log-linear}, where ABPG~\cite{Hanzely2021} based on TSE fails to accelerate. We do not claim GCS is more general than TSE, or that Acc-MD is always better than ABPG. Instead, they address {\it different structure regimes}. See Section~\ref{sec:A2} for discussion.

For the convex case $\mu=0$, we use a perturbation-homotopy argument \cite{chen2025accelerated} to derive the optimal accelerated rate $\mathcal{O}(C_{f,\phi}/k^{2})$ under GCS. For composite problems $\min_x f(x)+g(x)$, our method leads to a Bregman proximal gradient variant that handles non-smooth $g$ using variable splitting in the gradient term.

We test the Acc-MD methods on smooth, non-smooth, and constrained convex problems. Results show benefits over existing mirror descent methods.

The paper is organized as follows. Section \ref{sec: Preliminaries} gives preliminaries. Section \ref{sec: mirror descent} reviews mirror descent and gives a proof for linear convergence in relative strongly convex cases. Section \ref{sec: acc-md} introduces the Acc-MD flow and proves accelerated linear convergence under GCS. Section \ref{sec: extension to convex opt} and Section \ref{sec: extension to composite opt} extend the methods to convex and composite cases. Section \ref{sec: numerical} reports numerical results.

\section{Preliminaries}\label{sec: Preliminaries}
In this section, we present useful results in convex optimization and duality theory; see standard textbooks, e.g., \cite{bertsekas2009convex,rockafellar1970convex}, for details. Let $V$ be a normed vector space and $V^*$ be its dual space. The duality pairing between $V^{\ast}$ and $V$ is denoted by $\langle \chi, x\rangle$ for $\chi \in V^{\ast}$ and $x \in V$. When $V$ is a Hilbert space, the Riesz representation theorem allows us to identify $V$ with $V^*$, in which case the duality pairing coincides with the inner product $(\cdot, \cdot)$.

For a continuously differentiable function $f \in \mathcal{C}^1(V)$, the {\it Bregman divergence} is
\begin{equation*}
	D_f(y, x) := f(y) - f(x) - \langle \nabla f(x), y - x \rangle. 
\end{equation*}
For $f \in \mathcal{C}^1(V)$, $f$ is convex if and only if $D_f(y, x) \geq 0$ for all $x,y\in V$. If $f$ is strictly convex, then $D_f(y, x) = 0$ if and only if $x = y$. In general, the Bregman divergence is not symmetric. Its symmetrization is 
$$
D_f(y,x) + D_f(x,y) = \langle \nabla f(y) - \nabla f(x), y - x \rangle.
$$

A key tool in the convergence analysis is the three-point identity of Bregman divergence~\cite{chen1993convergence}:
\begin{equation}\label{eq:Bregmanidentity}
	\langle \nabla f(y) - \nabla f(x), y - z \rangle = D_f(y, x) + D_f(z, y) - D_f(z, x).
\end{equation}

We fix a smooth, strictly convex function of Legendre type $\phi \in \mathcal{C}^1(V)$, known as the {\it mirror function}. Let $\phi^*$ be the convex conjugate of $\phi$. The mappings between primal and dual variables are
$$
\chi = \nabla \phi(x) \Leftrightarrow x = \nabla \phi^*(\chi), \qquad \eta = \nabla \phi(y) \Leftrightarrow y = \nabla \phi^*(\eta),
$$
where $(x, y)$ are primal variables and $(\chi, \eta)$ are dual variables. The maps $\nabla \phi: V \to V^*$ and $\nabla \phi^*: V^* \to V$ are assumed to be one-to-one and easy to compute. We refer to $\phi$ as the mirror function and $\nabla \phi$ as the mirror map.

A symmetry relation connects the Bregman divergences of $\phi$ and $\phi^*$:
\begin{equation}\label{eq:DfDf*}
	D_\phi(x, y) = D_{\phi^*}(\eta, \chi),
\end{equation}
with reversed arguments. Moreover, the gradient of the Bregman divergence with respect to the first variable satisfies
\begin{equation}\label{eq:gradD}
	\nabla D_f(\cdot, x) = \nabla f(\cdot) - \nabla f(x), \qquad 
	\nabla D_{\phi^*}(\cdot, \chi) = \nabla \phi^*(\cdot) - \nabla \phi^*(\chi).
\end{equation}

Let $A$ be a self-adjoint, positive definite operator on a Hilbert space $V$ with inner product $(\cdot, \cdot)$. Then $(x, y)_A := (Ax, y)$ defines a new inner product with norm $\|\cdot\|_A$. The dual norm is $\|\cdot\|_{A^{-1}}$. The convexity and smoothness constants of a differentiable function $f$ relative to $\|\cdot\|_A$ satisfy
\begin{equation*}
	\mu_f(A)\, \|x - y\|_A^2 \leq \langle \nabla f(x) - \nabla f(y), x - y \rangle \leq L_f(A)\, \|x - y\|_A^2 \quad \forall\, x, y \in V.
\end{equation*}
When $A$ is the identity, $L_f$ and $\mu_f$ are standard constants in the Euclidean norm. A proper choice of $A$ may reduce the condition number $\kappa_A(f) := L_f(A) / \mu_f(A)$; in this case $A$ is a {\it preconditioner}. Similarly, the mirror function $\phi$ acts as a {\it nonlinear preconditioner}. This is useful when $L_f$ is large but the relative smoothness $L$ is small.

For two symmetric operators $D$ and $A$, we write $A < D$ if $D-A$ is positive definite and $A \leq D$ if $D-A$ is positive semi-definite. Note that $A \leq D$ is equivalent to $\lambda_{\max}(D^{-1}A) \leq 1$ for non-singular symmetric operators.
\section{Mirror Descent Methods}\label{sec: mirror descent}
Given an initial condition $x(0) = x_0$, the mirror descent flow is defined by the evolution of the mirror map:
\begin{equation}\label{eq:mdflow}
	\frac{\dd }{\dd t}\nabla\phi(x(t)) = -\nabla f(x(t)).
\end{equation}
The flow \eqref{eq:mdflow} can be expressed in the dual form \cite{krichene2015accelerated}:
\begin{equation}\label{eq:dualmd}
    \chi' = -\nabla f(x),
\end{equation}
where $\chi = \nabla \phi(x)$ represents the dual variable and $x = \nabla \phi^*(\chi)$. This formulation is closely related to the Bregman Inverse Scale Space dynamics studied in \cite{Osher2016}.

Discretizing \eqref{eq:mdflow} or \eqref{eq:dualmd} using the explicit Euler method with step size $\alpha_k > 0$ yields the iteration
\begin{equation}\label{eq:md}
	\nabla\phi(x_{k+1}) - \nabla \phi(x_k) = -\alpha_k \nabla f(x_k).
\end{equation}
By the definition of the Bregman divergence, this update is equivalent to the standard mirror descent step:
\begin{equation*}
	x_{k+1} = \arg\min_{x \in V} \left\{
	\langle \nabla f(x_k), x \rangle + \frac{1}{\alpha_k} D_\phi(x, x_k)
	\right\}.
\end{equation*}

The following three-point identity simplifies the convergence analysis by connecting the geometries of the mirror function $\phi$ and the objective $f$.

\begin{lemma}
Let $\{x_k\}$ be the sequence generated by the mirror descent method \eqref{eq:md}. Then
\begin{equation}\label{eq:mdidentity}
\begin{aligned}
&	D_\phi(x^{\star}, x_{k+1}) - D_\phi(x^{\star}, x_k) + D_\phi(x_{k+1}, x_k) \\
	={}& \alpha_k \Big[ - D_f(x_{k+1}, x^{\star}) - D_f(x^{\star}, x_k) + D_f(x_{k+1}, x_k) \Big].
\end{aligned}
\end{equation}
\end{lemma}
\begin{proof}
Using the three-point identity \eqref{eq:Bregmanidentity}, the optimality condition $\nabla f(x^\star) = 0$, and the update \eqref{eq:md}, we have:
\begin{align*}
	&D_\phi(x^{\star}, x_{k+1}) - D_\phi(x^{\star}, x_k) + D_\phi(x_{k+1}, x_k) \\
={}& \langle \nabla \phi(x_{k+1}) - \nabla \phi(x_k), x_{k+1} - x^{\star} \rangle  \\
	={}& -\alpha_k \langle \nabla f(x_k) - \nabla f(x^{\star}), x_{k+1} - x^{\star} \rangle\\ ={}&\alpha_k \left[ - D_f(x_{k+1}, x^{\star}) - D_f(x^{\star}, x_k) + D_f(x_{k+1}, x_k) \right].
\end{align*}
\end{proof}

To establish linear convergence without assuming Euclidean Lipschitz continuity of $\nabla f$, we utilize the notions of relative smoothness and relative strong convexity \cite{lu2018relatively}.

\begin{assumption}\label{assump: A1}
 The function $f$ is {\it $L$-relatively smooth} and {\it $\mu$-relatively convex} with respect to $\phi$ if there exist $\mu \geq 0$ and $L \geq \mu$ such that
\begin{equation}\tag{A1}\label{eq: A1}
	\mu D_\phi(x, y) \leq D_f(x, y) \leq L D_\phi(x, y) \quad \forall\, x, y \in V.
\end{equation}
\end{assumption}

Equivalently, \cref{assump: A1} holds if and only if $L\phi - f$ and $f - \mu \phi$ are convex. 

\begin{theorem}
Suppose $f$ satisfies \cref{assump: A1} with $L \geq \mu > 0$. Let $\{x_k\}$ be the sequence from \eqref{eq:md}. For any step size $\alpha_k \leq 1/L$, the Bregman divergence satisfies the following decay:
\begin{equation}
 D_\phi(x^{\star}, x_k) - D_\phi(x^{\star}, x_{k+1}) \geq \alpha_k \big[ D_f(x_{k+1}, x^{\star}) + D_f(x^{\star}, x_k)\big]. 
\end{equation}
Specifically, if $\alpha_k = 1/L$, mirror descent achieves a linear convergence rate:
\begin{equation}\label{eq:mdrate}
	D_\phi(x^{\star}, x_k) \leq \left(1 - \frac{\mu}{L}\right)^k D_\phi(x^{\star}, x_0).
\end{equation}
\end{theorem}
\begin{proof}
Combining \eqref{eq:mdidentity} with the upper bounds in \eqref{eq: A1}, we obtain:
\begin{align*}
&	D_\phi(x^{\star}, x_{k+1}) - D_\phi(x^{\star}, x_k) \\
	=& -\alpha_k \big [ D_f(x_{k+1}, x^{\star}) + D_f(x^{\star}, x_k)\big ]  + \left [\alpha_k D_f (x_{k+1}, x_k) - D_\phi(x_{k+1}, x_k)\right ] \\
	\leq &-\alpha_k \big [ D_f(x_{k+1}, x^{\star}) + D_f(x^{\star}, x_k)\big ].
\end{align*}
Applying the relative strong convexity $D_f(x^\star, x_k) \ge \mu D_\phi(x^\star, x_k)$ and setting $\alpha_k = 1/L$ yields:
$$D_\phi(x^{\star}, x_{k+1}) \le (1 - \mu/L) D_\phi(x^{\star}, x_k).$$
Iterating the inequality completes the proof.
\end{proof}

\section{Accelerated Mirror Descent Methods}\label{sec: acc-md}
We use the variable and operator splitting (VOS) framework~\cite{chen2025accelerated} to develop accelerated mirror descent methods. Throughout this section, we assume $f$ satisfies \cref{assump: A1} with $L \geq \mu > 0$.

\subsection{Flow and stability}
By introducing dual variables $\chi = \nabla \phi(x)$ and $\eta = \nabla \phi(y)$, we write the accelerated mirror descent flow~\eqref{eq:VOSflow} as
\begin{equation}\label{eq:VOSdualflow}
x' = y - x, \quad \eta' = -\mu^{-1} \nabla \fb(x) - \eta,
\end{equation}
where the $\mu$-shifted function is defined as
$$
\fb := f - \mu \phi.
$$
Under \cref{assump: A1}, $\fb$ is convex. We define a Lyapunov function:
\begin{equation} \label{eq:lya_func}
    \mathcal{E}(x, \eta) := D_{\fb}(x, x^{\star}) + \mu D_{\phi^*}(\eta, \chi^{\star}),
\end{equation}
where $x^{\star}$ is a minimizer of $f$ and $\chi^{\star} = \nabla \phi(x^{\star})$. This energy couples primal and dual Bregman divergences to capture the geometry of the mirror map. We show exponential stability by verifying a strong Lyapunov condition~\cite{chen2021unified}.

\paragraph{Strong Lyapunov condition for linear cases}
We illustrate the flow with a linear system $Ax^*=b$. Let $A$ and $B$ be two symmetric positive definite (SPD) matrices satisfying $A \ge \mu B^{-1}$. Define
$$
\nabla \fb(x) = A_{-\mu}x - b := (A - \mu B^{-1})x - b, \quad \nabla \phi(x) = B^{-1}x.
$$
The flow simplifies to
\begin{equation}\label{eq:linear}
x' = y - x, \qquad y' = -\mu^{-1} B A_{-\mu} x - y = -\mu^{-1} B (A x - b) + x - y.
\end{equation}

Let $\boldsymbol z = (x, y)^{\top}$ and introduce the matrices
$$
\mathcal D = \begin{pmatrix} A_{-\mu} & 0 \\ 0 & \mu B^{-1} \end{pmatrix}, \quad 
\mathcal G = \begin{pmatrix} -I & I \\ -\mu^{-1} B A_{-\mu} & -I \end{pmatrix}.
$$
The product $\mathcal D\mathcal G$ is
$$
\mathcal D\mathcal G = \begin{pmatrix} -A_{-\mu} & A_{-\mu} \\ -A_{-\mu} & -\mu B^{-1} \end{pmatrix}.
$$
We write the Lyapunov function and the flow as
$$
\mathcal E(\boldsymbol z) = \frac{1}{2}\|\boldsymbol z - \boldsymbol z^\star\|_{\mathcal D}^2, \quad \boldsymbol z' = \mathcal G(\boldsymbol z - \boldsymbol z^\star),
$$
where $\boldsymbol z^\star = (x^{\star}, x^{\star})^{\top}$ is the equilibrium point. Then
$$
\nabla \mathcal E(\boldsymbol z) \cdot \mathcal G(\boldsymbol z) = (\boldsymbol z - \boldsymbol z^\star)^{\top} \mathcal D \mathcal G (\boldsymbol z - \boldsymbol z^\star).
$$
Since $x^{\top} M x = \frac{1}{2} x^{\top}(M + M^{\top})x$, we have
$$
\frac{1}{2} \bigl(\mathcal D \mathcal G + (\mathcal D \mathcal G)^{\top}\bigr) = -\mathcal D \quad \Longrightarrow \quad -\nabla \mathcal E(\boldsymbol z) \cdot \mathcal G(\boldsymbol z) = 2 \mathcal E(\boldsymbol z).
$$

\paragraph{Strong Lyapunov condition of nonlinear cases}
We now consider the nonlinear case. Due to the non-symmetry of the Bregman divergence, the factor $2$ becomes $1$. 

\begin{lemma}
Suppose $f$ is $\mu$-relative convex with respect to $\phi$ with $\mu > 0$. Let $\mathcal{E}(x, \eta)$ be defined by~\eqref{eq:lya_func}, and define the vector field $\mathcal{G}(x, \eta) = (y - x,\, -\mu^{-1} \nabla \fb(x) - \eta)$, where the dual variables satisfy $\chi = \nabla \phi(x)$ and $\eta = \nabla \phi(y)$. Then the following strong Lyapunov property holds:
\begin{equation}\label{eq:stLy}
    -\nabla \mathcal{E}(x, \eta) \cdot \mathcal{G}(x, \eta)
    = \mathcal{E}(x, \eta) + D_{\fb}(x^{\star}, x) + \mu D_{\phi}(y, x^{\star}).
\end{equation}
Consequently, any solution $(x(t), y(t))$ of the flow~\eqref{eq:VOSflow} satisfies the exponential decay bound
\begin{equation}\label{eq:stability}
    \mathcal{E}(x(t), \eta(t)) \leq e^{-t} \mathcal{E}(x(0), \eta(0)).
\end{equation}
\end{lemma}

\begin{proof}
Using~\eqref{eq:gradD}, we compute the gradients of the Lyapunov function:
$$
\partial_x \mathcal{E} = \nabla \fb(x) - \nabla \fb(x^{\star}), \quad
\partial_\eta \mathcal{E} = \mu \left( \nabla \phi^*(\eta) - \nabla \phi^*(\chi^{\star}) \right) = \mu (y - x^{\star}).
$$
By direct calculation and the symmetry relation~\eqref{eq:DfDf*}, we have
\begin{align*}
 -\nabla \mathcal{E}(x, \eta) \cdot \mathcal{G}(x, \eta)
&= \langle \nabla \fb(x) - \nabla \fb(x^{\star}), x - y \rangle \\
&+ \mu \left\langle y - x^{\star},\, \mu^{-1} (\nabla \fb(x) - \nabla \fb(x^{\star})) + \eta - \chi^{\star} \right\rangle \\
&= \langle \nabla \fb(x) - \nabla \fb(x^{\star}), x - x^{\star} \rangle
+ \mu \langle \nabla \phi^*(\eta) - \nabla \phi^*(\chi^{\star}), \eta - \chi^{\star} \rangle \\
&= D_{\fb}(x, x^{\star}) + D_{\fb}(x^{\star}, x) + \mu D_{\phi^*}(\chi^{\star}, \eta) + \mu D_{\phi^*}(\eta, \chi^{\star}) \\
&= \mathcal{E}(x, \eta) + D_{\fb}(x^{\star}, x) + \mu D_\phi(y, x^{\star}).
\end{align*}
Since $D_{\fb}(x^{\star}, x) \geq 0$ and $D_\phi(y, x^{\star}) \geq 0$ follow from the convexity of $\fb$ and $\phi$, we obtain the inequality
$$
\nabla \mathcal{E}(x, \eta) \cdot \mathcal{G}(x, \eta) \leq - \mathcal{E}(x, \eta).
$$
The exponential decay \eqref{eq:stability} follows from Gr{\"o}nwall's inequality.
\end{proof}

As a result, we have $x(t), y(t) \in B(x^{\star}, R)$ for all $t \geq 0$, where $B(x^{\star}, R)$ is the ball of radius $R$ centered at $x^{\star}$.

\subsection{Accelerated Mirror Descent Methods}
We propose an accelerated mirror descent (Acc-MD) method via an implicit--explicit discretization of~\eqref{eq:VOSflow}:
\begin{subequations}\label{Acc-MD}
\begin{align}
\label{eq:AGD1}
&\frac{x_{k+1}-x_k}{\alpha}= y_k - x_{k+1}, \\
\label{eq:AGD2}
&\frac{\nabla \phi(y_{k+1})-\nabla \phi(y_k)}{\alpha}
= - \frac{1}{\mu}\bigl(2\nabla \fb(x_{k+1})-\nabla \fb(x_k)\bigr) - \nabla \phi(y_{k+1}).
\end{align}
\end{subequations}
In~\eqref{eq:AGD1}, $y$ is discretized explicitly as $y_k$ to update $x_{k+1}$, whereas in~\eqref{eq:AGD2} the term $-\nabla\phi(y)$ is implicit. The extrapolation
$$
\nabla \fb(x)\;\approx\;2\nabla \fb(x_{k+1})-\nabla \fb(x_k)
$$
acts as an accelerated over-relaxation (AOR), symmetrizing the error equation. This mechanism was used in~\cite{wei2024accelerated} to develop the first globally convergent accelerated Heavy-Ball method (AOR-HB).

The VOS framework~\cite{chen2025accelerated} is not naturally built for mirror geometry. To extend it beyond the Euclidean setting, we introduce a new geometric assumption.

\begin{assumption}\label{assump: A2}
For $f, \phi \in \mathcal C^1(V)$, there exists a constant $C_{f,\phi}>0$ such that the following generalized Cauchy--Schwarz (GCS) inequality holds:
\begin{equation}\tag{A2}\label{eq: A2}
\bigl| \langle \nabla \fb(x)-\nabla \fb(\hat x),\, y-\hat y \rangle \bigr|
\le 2\sqrt{C_{f,\phi}}\,
D_{\fb}^{1/2}(x,\hat x)\, D_{\phi}^{1/2}(\hat y,y),
\quad \forall x,\hat x,y,\hat y \in V.
\end{equation}
\end{assumption}

Briefly, \cref{assump: A1} controls distances, while \cref{assump: A2} imposes an angle condition. Using dual variables $\zeta = \nabla \fb(x)$ and the identity $D_{\fb}(x,\hat x)=D_{\fb^*}(\hat \zeta, \zeta)$, we rewrite the inequality as
\begin{equation}\label{eq:CSA2}
\frac{\bigl| \langle \zeta - \hat \zeta,\, y-\hat y \rangle \bigr|}
{\sqrt{2}\, D_{\fb^*}^{1/2}(\hat \zeta,\zeta)\, \sqrt{2}\, D_{\phi}^{1/2}(\hat y,y)}
\le \sqrt{C_{f,\phi}},
\qquad \forall x,\hat x,y,\hat y \in V.
\end{equation}
In the quadratic case, $\sqrt{2}\, D_{\phi}^{1/2}(\hat y,y)=\|\hat y-y\|_{\nabla^2\phi}$. Inequality~\eqref{eq:CSA2} therefore measures the angle between vectors under different geometries.

We summarize~\eqref{Acc-MD} in \cref{alg:AccMD} and call it the {\it forward Acc-MD} method.

\begin{algorithm}
\caption{Forward Accelerated Mirror Descent (Acc-MD)}
\label{alg:AccMD}
\linespread{1.325}\selectfont
\begin{algorithmic}[1]
\STATE \textbf{Input:} $x_0, y_0 \in \mathbb{R}^n$, parameters $\mu>0$, $C_{f,\phi}>0$
\STATE Set $\alpha = \sqrt{\mu/C_{f,\phi}}$
\FOR{$k=0,1,2,\dots$}
\STATE $\displaystyle x_{k+1} = \frac{1}{1+\alpha}\bigl(x_k + \alpha y_k\bigr)$
\STATE $\displaystyle
y_{k+1}
=
\arg\min_{y\in\mathbb R^n}
(1+\alpha)\phi(y)
-
\left\langle
-\frac{\alpha}{\mu}\bigl(2\nabla \fb(x_{k+1})-\nabla \fb(x_k)\bigr)
+\nabla\phi(y_k),\, y
\right\rangle$
\ENDFOR
\end{algorithmic}
\end{algorithm}

We also consider an alternative discretization:
\begin{subequations}\label{Acc-MD-alternative}
\begin{align}
\label{eq:AGD1-alternative}
&\frac{x_{k+1}-x_k}{\alpha}= 2y_{k+1}-y_k-x_{k+1}, \\
\label{eq:AGD2-alternative}
&\frac{\nabla\phi(y_{k+1})-\nabla\phi(y_k)}{\alpha}
=
-\frac{1}{\mu}\nabla f(x_k)
+\nabla\phi(x_k)
-\nabla\phi(y_{k+1}).
\end{align}
\end{subequations}
In this scheme, the gradient term $\nabla \fb(x)$ is explicit in~\eqref{eq:AGD2-alternative}, while $y$ in~\eqref{eq:AGD1-alternative} uses AOR. Since the update computes $y_{k+1}$ then $x_{k+1}$, we call this the {\it backward Acc-MD} method; see \cref{alg:AccMDback}.

\begin{algorithm}
\caption{Backward Accelerated Mirror Descent}
\label{alg:AccMDback}
\linespread{1.325}\selectfont
\begin{algorithmic}[1]
\STATE \textbf{Input:} $x_0, y_0 \in \mathbb{R}^n$, parameters $\mu>0$, $C_{f,\phi}>0$
\STATE Set $\alpha = \sqrt{\mu/C_{f,\phi}}$
\FOR{$k=0,1,2,\dots$}
\STATE $\displaystyle
y_{k+1}
=
\arg\min_{y\in\mathbb R^n}
(1+\alpha)\phi(y)
-
\left\langle
-\frac{\alpha}{\mu}\nabla f(x_k)
+\alpha\nabla\phi(x_k)
+\nabla\phi(y_k),\, y
\right\rangle$
\STATE $\displaystyle
x_{k+1}
=
\frac{1}{1+\alpha}\bigl [x_k+\alpha\bigl(2y_{k+1}-y_k\bigr)\bigr]$
\ENDFOR
\end{algorithmic}
\end{algorithm}

One benefit of the forward Acc-MD scheme is that $x_{k+1}$ is a convex combination of $x_k$ and $y_k$. Thus, if $x_k,y_k$ stay in a convex set $K$, then $x_{k+1}$ does too. This property is useful when $K$ is a constraint set.

\subsection{Convergence analysis}
We establish the accelerated linear convergence result for the Acc-MD methods under relative strong convexity $\mu > 0$. We first use linear cases to illustrate the proof.

\paragraph{Convergence analysis of linear cases}
Introduce the matrices
$$
\mathcal A^{\text{sym}}=
\begin{pmatrix}
0 & A_{-\mu}\\
A_{-\mu} & 0
\end{pmatrix},
\quad
\mathcal E(\boldsymbol z)=\frac12\|\boldsymbol z-\boldsymbol z^\star\|_{\mathcal D}^2,
\quad
\mathcal E^\alpha(\boldsymbol z)=\frac12\|\boldsymbol z-\boldsymbol z^\star\|_{\mathcal D+\alpha\mathcal A^{\text{sym}}}^2 .
$$
Spectral analysis shows that $\mathcal D+\alpha\mathcal A^{\text{sym}}> 0$ if
$$
|\alpha|\le
\sqrt{\mu/\lambda_{\max}\!\bigl(BA_{-\mu}\bigr)}
< \sqrt{\frac{\mu}{L-\mu}} .
$$
We take $\alpha=\sqrt{\mu/L}$, so that $\mathcal D\pm\alpha\mathcal A^{\text{sym}}$ are symmetric positive definite (SPD). Since $\mathcal E$ is quadratic, we have
\begin{equation}\label{eq:Edifference}
\mathcal E(\boldsymbol z_{k+1})-\mathcal E(\boldsymbol z_k)
= \langle \nabla\mathcal E(\boldsymbol z_{k+1}),\,\boldsymbol z_{k+1}-\boldsymbol z_k \rangle
-\frac12\|\boldsymbol z_{k+1}-\boldsymbol z_k\|_{\mathcal D}^2 .
\end{equation}
The Acc-MD method corrects the implicit Euler discretization. Let $\Delta\boldsymbol z_k=\boldsymbol z_{k+1}-\boldsymbol z_k$. We write
$$
\boldsymbol z_{k+1}-\boldsymbol z_k
=\alpha\mathcal G(\boldsymbol z_{k+1})-\alpha\mathcal C\,\Delta\boldsymbol z_k
:=\alpha\mathcal G(\boldsymbol z_{k+1})
-\alpha
\begin{pmatrix}
0 & I\\
\mu^{-1}BA_{-\mu} & 0
\end{pmatrix}
(\boldsymbol z_{k+1}-\boldsymbol z_k).
$$
Expanding the cross term gives
$$
\begin{aligned}
\langle \nabla \mathcal E(\boldsymbol z_{k+1}),\,\mathcal C \Delta\boldsymbol z_k \rangle
={}& \langle \boldsymbol z_{k+1}-\boldsymbol z^\star,\,\Delta\boldsymbol z_k \rangle_{\mathcal D\mathcal C}
= \langle \boldsymbol z_{k+1}-\boldsymbol z^\star,\,\boldsymbol z_{k+1}-\boldsymbol z_k \rangle_{\mathcal A^{\text{sym}}}\\
={}&\frac12\|\boldsymbol z_{k+1}-\boldsymbol z^\star\|_{\mathcal A^{\text{sym}}}^2
+\frac12\|\boldsymbol z_{k+1}-\boldsymbol z_k\|_{\mathcal A^{\text{sym}}}^2-\frac12\|\boldsymbol z_k-\boldsymbol z^\star\|_{\mathcal A^{\text{sym}}}^2.
\end{aligned}
$$
Using the strong Lyapunov property $\nabla \mathcal E(\boldsymbol z_{k+1})\cdot \mathcal G(\boldsymbol z_{k+1})= - 2\,\mathcal E(\boldsymbol z_{k+1})$, we obtain
\begin{equation}\label{eq:decay linear}
\mathcal E^\alpha(\boldsymbol z_{k+1})-\mathcal E^\alpha(\boldsymbol z_k)
=-\alpha\mathcal E^\alpha(\boldsymbol z_{k+1})
-\frac{\alpha}{2}\|\boldsymbol z_{k+1}-\boldsymbol z^\star\|_{\mathcal D-\alpha\mathcal A^{\text{sym}}}^2
-\frac12\| \Delta \boldsymbol z_k\|_{\mathcal D+\alpha\mathcal A^{\text{sym}}}^2 .
\end{equation}
Dropping negative terms yields linear convergence:
$$
\mathcal E^\alpha(\boldsymbol z_{k+1})-\mathcal E^\alpha(\boldsymbol z_k)
\le -\alpha\mathcal E^\alpha(\boldsymbol z_{k+1})
\quad \Longrightarrow \quad
\mathcal E^\alpha(\boldsymbol z_{k+1})\le (1+\alpha)^{-1}\mathcal E^\alpha(\boldsymbol z_k).
$$
The AOR term associates the cross term with $\mathcal A^{\text{sym}}$, allowing the identity of squares to control the cross term. 

\paragraph{Convergence analysis of nonlinear cases}
For $(x, \hat x, y, \hat y)$ and $\alpha \in \mathbb R$, we define the cross term
$$
\mathcal B^\alpha(x, \hat x, \hat y, y) 
    = D_{\fb}(x, \hat x) + \mu D_{\phi}(\hat y, y) 
      + \alpha \langle \nabla \fb(x) - \nabla \fb(\hat x),\, y - \hat y \rangle.
$$
We define the modified Lyapunov function
\begin{equation}\label{mod_lya}
\mathcal{E}^\alpha(x,y):=\mathcal{B}^\alpha(x,x^{\star}, x^{\star},y)= \mathcal{E}(x,\eta(y)) +\alpha\langle \nabla \fb(x)-\nabla \fb(x^{\star}),y-x^{\star} \rangle.
\end{equation}
We now generalize identity \eqref{eq:decay linear}.

\begin{lemma}\label{lem:identity}
    Let $(x_k,y_k)$ be the sequence from the Acc-MD method \eqref{Acc-MD}. Then
\begin{equation}\label{eq:Epha}
\begin{aligned}
\mathcal{E}^\alpha(x_{k+1},y_{k+1})-\mathcal{E}^\alpha(x_k,y_k) ={}& - \alpha \mathcal{E}^\alpha(x_{k+1},y_{k+1}) \\
&- \alpha \, \mathcal B^{-\alpha}(x^{\star}, x_{k+1},y_{k+1}, x^{\star}) -   \mathcal B^{\alpha}(x_k, x_{k+1},y_{k+1}, y_k).
\end{aligned}
\end{equation}
\end{lemma}

\begin{proof}
Let $\boldsymbol z = (x,\eta)$. We expand the difference $\mathcal E(\boldsymbol z)$ at $\boldsymbol z_{k+1}$:
\begin{align*}
\mathcal E(\boldsymbol z_{k+1}) - \mathcal E(\boldsymbol z_k)={}&\langle \nabla \mathcal E(\boldsymbol z_{k+1}),\boldsymbol z_{k+1}-\boldsymbol z_k
\rangle -D_{\mathcal E}(\boldsymbol z_k,\boldsymbol z_{k+1}) \\
={}& \alpha\langle \nabla \mathcal E(\boldsymbol z_{k+1}),\mathcal{G}(\boldsymbol z_{k+1}) \rangle - D_{\mathcal E}(\boldsymbol z_k,\boldsymbol z_{k+1}) \\
&- \alpha\langle \nabla \fb(x_{k+1})-\nabla \fb(x^{\star}), y_{k+1}-y_k \rangle \\
&- \alpha\langle y_{k+1} - x^{\star}, \nabla \fb(x_{k+1})-\nabla \fb(x_k) \rangle.
\end{align*}
The difference from the implicit Euler discretization is symmetrized as it can be written as
\begin{equation*}
\begin{aligned}
& -\alpha \langle \nabla \fb(x_{k+1})-\nabla \fb(x^{\star}),y_{k+1}-x^{\star} \rangle + \alpha  \langle \nabla \fb(x_{k})-\nabla \fb(x^{\star}), y_{k}-x^{\star} \rangle \\
&- \alpha \langle \nabla \fb(x_{k+1})-\nabla \fb(x_k), y_{k+1} - y_k \rangle.
\end{aligned}
\end{equation*}
Using identity (\ref{eq:stLy}) to expand $\langle \nabla \mathcal E(\boldsymbol z_{k+1}),\mathcal{G}(\boldsymbol z_{k+1}) \rangle$ and rearranging gives
\begin{align*}
 &     \mathcal{E}^\alpha(x_{k+1},y_{k+1})-\mathcal{E}^\alpha(x_k,y_k) \\
      ={}& -\alpha \mathcal E(x_{k+1}, \eta_{k+1}) - \alpha D_{\fb} (x^{\star}, x_{k+1}) - \mu \alpha D_{\phi}(y_{k+1}, x^{\star}) \\
      & - D_{\mathcal E}(\boldsymbol z_k,\boldsymbol z_{k+1}) - \alpha \langle \nabla \fb(x_{k+1})-\nabla \fb(x_k), y_{k+1} - y_k \rangle \\
      ={}& -\alpha \mathcal E^{\alpha}(x_{k+1}, y_{k+1}) - \alpha \, \mathcal B^{-\alpha}(x^{\star}, x_{k+1},y_{k+1}, x^{\star}) - \mathcal B^{\alpha}(x_k, x_{k+1},y_{k+1}, y_k).
\end{align*}
\end{proof}

\cref{assump: A2} is essential both for defining the algorithm and for establishing its convergence.

\begin{lemma}\label{lem:crossterm}
Suppose \cref{assump: A2} holds. Then for $|\alpha| \leq \sqrt{\mu / C_{f,\phi}}$, we have $\mathcal B^{\alpha}(x, \hat x, \hat y, y) \geq 0.$
\end{lemma}
\begin{proof}
By \eqref{eq: A2} and the inequality $2ab\leq a^2+b^2$, we have
$$
\begin{aligned}
\alpha | \langle \nabla \fb(x) - \nabla \fb(\hat x), y - \hat y \rangle | &\leq 2\alpha \sqrt{C_{f,\phi}} \, D_{\fb}^{1/2}(x,\hat x) D_{\phi}^{1/2}(\hat y, y) \\
&\leq \alpha \sqrt{C_{f,\phi}/\mu} \left(D_{\fb}(x,\hat x)+\mu D_{\phi}(\hat y, y)\right).
\end{aligned}
$$
Therefore $\mathcal B^{\alpha}(x, \hat x, \hat y, y) \geq 0$ if $\alpha \leq \sqrt{\mu/C_{f,\phi}}$.
\end{proof}

As shown in \cref{lem:crossterm}, when $\alpha \leq \sqrt{\mu/C_{f,\phi}}$, the modified energy $\mathcal{E}^\alpha(x,y) \geq 0$ is a valid Lyapunov function. 

\begin{theorem}[Convergence of forward Acc-MD method]\label{thm:convergence rate of Acc-MD dual}
Suppose $f$ is $\mu$-relatively convex with respect to $\phi$ with $\mu > 0$ and \cref{assump: A2} holds with constant $C_{f,\phi}$. Let $(x_k,y_k)$ be the sequence from scheme \eqref{Acc-MD} with initial value $(x_0, y_0)$, $\eta_k = \nabla \phi(y_k)$, and step size $\alpha = \sqrt{\mu/C_{f,\phi}}$. Then there exists a constant $C_0$ such that the following accelerated linear convergence holds:
\begin{equation}\label{eq: linear conv Acc-MD}
D_{\fb}(x_{k+1}, x^{\star}) + \mu D_{\phi^*}(\eta_{k+1}, \chi^{\star}) \leq C_0 \left( \frac{1}{1+\sqrt{\mu/C_{f,\phi}}} \right)^k, \quad k \geq 1.
\end{equation}
\end{theorem}
\begin{proof}
By \cref{lem:crossterm}, for $\alpha = \sqrt{\mu/C_{f,\phi}}$, we drop non-positive terms from identity \eqref{eq:Epha} to obtain
\begin{equation}\label{eq:Ealphaconvergence}
\mathcal{E}^\alpha(x_{k+1},y_{k+1}) \leq \frac{1}{1+\alpha}\mathcal{E}^\alpha(x_{k},y_{k}) \leq \left( \frac{1}{1+\sqrt{\mu/C_{f,\phi}}} \right)^{k+1} \mathcal{E}^\alpha(x_{0},y_{0}).
\end{equation}
From the proof of \cref{lem:identity}, we have 
$$
\alpha \mathcal E(x_{k+1}, \eta_{k+1}) \leq \mathcal{E}^\alpha(x_k,y_k) - \mathcal{E}^\alpha(x_{k+1},y_{k+1}) \leq \left( \frac{1}{1+\sqrt{\mu/C_{f,\phi}}} \right)^{k} \mathcal{E}^\alpha(x_{0},y_{0}),
$$
which yields \eqref{eq: linear conv Acc-MD} with $C_0 = \mathcal{E}^\alpha(x_0,y_0) \sqrt{C_{f,\phi}/\mu}$. 
\end{proof}

The convergence of the backward Acc-MD scheme \eqref{Acc-MD-alternative} can be proved by changing the sign of the Lyapunov function, $\mathcal E^\alpha \to \mathcal E^{-\alpha}$, and repeating the procedure.



If $C_{f,\phi} = c L$, then one exactly recovers the standard accelerated convergence rate like $1 - \sqrt{1/(c \kappa)}$ but the condition number is relative to the mirror function showing the preconditioning effect. 

\subsection{Discussion on the generalized Cauchy--Schwarz condition}\label{sec:A2}
Without extra structural conditions, mirror descent methods under \cref{assump: A1} are limited to $\mathcal{O}(1/k)$ convergence; see~\cite{DragomTaylordAspreBolte2022Optimal}. Consequently, accelerated mirror descent is generally impossible without assumptions beyond~\eqref{eq: A1}. Additional structure, such as the triangle scaling exponent (TSE), the generalized Cauchy--Schwarz (GCS) condition~\eqref{eq: A2}, or related variants \cite{bauschke2019linear,mukkamala2020convex}, is necessary for theoretical guarantees and algorithmic acceleration.

We first show that the GCS angle condition implies relative smoothness.
\begin{lemma}
Suppose \eqref{eq: A2} holds with constant $C_{f,\phi}$. Then $f$ is relatively smooth with respect to $\phi$ with $L \leq 4C_{f,\phi} + \mu$.
\end{lemma}
\begin{proof}
Set $y = x$ and $\hat y = \hat x$ in \eqref{eq: A2}. Using the property $\langle \nabla \fb(x) - \nabla \fb(\hat x), x - \hat x \rangle = D_{\fb}(x, \hat x) + D_{\fb}(\hat x, x) \geq D_{\fb}(x, \hat x)$, the GCS inequality becomes:
$$
D_{\fb}(x, \hat x) \leq 2\sqrt{C_{f,\phi}}\, D_{\fb}^{1/2}(x, \hat x) D_{\phi}^{1/2}(\hat x, x).
$$
The desired result then follows through simple algebraic manipulation.
\end{proof}
The constant $4C_{f,\phi}$ in the global bound is a consequence of the potential asymmetry of the Bregman divergence. However, this can be sharpened locally for $\mathcal{C}^2$ functions. By observing that $\langle \nabla \fb(x) - \nabla \fb(\hat x), x - \hat x \rangle = 2D_{\fb}(x, \hat x) + (D_{\fb}(\hat x, x) - D_{\fb}(x, \hat x))$, and noting that the difference $D_{\fb}(\hat x, x) - D_{\fb}(x, \hat x) = o(\|x - \hat x\|^2)$ vanishes faster than the divergence itself, we find that:
$$
D_{\fb}(x, \hat x) \leq (C_{f,\phi} + o(1)) D_{\phi}(\hat x, x) \quad \text{as } x \to \hat x.
$$
This implies that locally, $L \approx C_{f,\phi} + \mu$, which is consistent with the condition number recovered in the quadratic case in \cref{thm:A2}.

We next verify \cref{assump: A2} in several common settings to show it is a natural and practical condition.

\begin{theorem}\label{thm:A2}
Let $A$ be a self-adjoint, positive definite operator. Suppose $\phi$ is $\mu_{\phi}$-strongly convex and $\fb := f - \mu \phi$ is $L_{\fb}$-smooth. Then
$$ C_{f,\phi} \leq \inf_{\text{SPD } A} \frac{L_{\fb}(A)}{\mu_{\phi}(A)} \leq \frac{L_{\fb}}{\mu_{\phi}}. $$
\end{theorem}
\begin{proof}
For any SPD $A$, we have
$$ \begin{aligned} \langle \nabla \fb(x) - \nabla \fb(\hat{x}),\, y - \hat{y} \rangle &= \langle A^{-1/2} (\nabla \fb(x) - \nabla \fb(\hat{x})),\, A^{1/2}(y - \hat{y}) \rangle \\ &\leq \| \nabla \fb(x) - \nabla \fb(\hat{x}) \|_{A^{-1}} \| y - \hat{y} \|_A. \end{aligned} $$
By the strong convexity of $\phi$ in the $\|\cdot\|_A$ norm and the co-convexity of $\fb$ in the $\|\cdot\|_{A^{-1}}$ norm, we obtain
$$ 2 D_\phi(\hat{y}, y) \geq \mu_{\phi}(A) \| y - \hat{y} \|_A^2, \quad 2 D_{\fb}(x, \hat{x}) \geq \frac{1}{L_{\fb}(A)} \| \nabla \fb(x) - \nabla \fb(\hat{x}) \|_{A^{-1}}^2. $$
Combining these inequalities yields the first estimate. Setting $A = I$ gives the second.
\end{proof}

The smoothness of $\fb=f-\mu\phi$ is less restrictive than the smoothness of $f$. We decompose the objective as
$$ f=\fb+\mu\phi. $$
The non-Lipschitz part of $\nabla f$ is absorbed into the mirror function, so the remaining term $\nabla \fb$ is smooth.

\begin{example}[Log-linear model]\label{ex: log-linear}\rm 
Consider the log-linear dual model in machine learning~\cite{Collins2008}:
$$ f(x) = \sum_{i=1}^d x_i \log x_i + \frac12\, x^\top(\mathbf g\mathbf g^\top)x, \quad \phi(x)=\sum_{i=1}^d x_i\log x_i, $$
where $x=(x_i)\in\Delta^{d-1}$ is the dual variable on the $(d-1)$-dimensional simplex and $\mathbf g\in\mathbb R^d$. The mirror function $\phi$ is the Shannon entropy. Setting $\mu=1$, we have $\fb(x)=\tfrac12\,x^\top(\mathbf g\mathbf g^\top)x$. Since $\mathbf g\mathbf g^\top$ is positive semidefinite, we regularize it by
$$ A=\mathbf g\mathbf g^\top+\varepsilon I,\qquad \varepsilon>0. $$
Under this regularization, the smoothness constant is $L_{\fb}(A)=1$. Collins et al. \cite[Lemma 7]{Collins2008} proved that $f$ is $1$-relatively strongly convex and $(1+|A|_\infty)$-relatively smooth with respect to $\phi$, where $| \mathbf{g}\mathbf{g}^\top |_\infty$ is the largest entry of $\mathbf{g}\mathbf{g}^\top$. Thus, \cref{assump: A1} holds with $\mu = 1$ and $L = |\mathbf{g}\mathbf{g}^\top|_\infty$.

The Bregman divergence $D_\phi(x, z)$ is the Kullback--Leibler (KL) divergence between two discrete probability measures: 
$$ D_\phi(x, z) = {\rm KL}(x, z) := \sum_i x_i \log\left(\frac{x_i}{z_i}\right). $$
By Pinsker's inequality, $\lVert y-\hat{y} \rVert_1 \leq \sqrt{2\,\mathrm{KL}(\hat{y},y)}$, which implies
\begin{equation}\label{eq:Pinsker}
\begin{aligned} \lVert y-\hat{y} \rVert^2_{A} \leq {}&(\lVert \mathbf{g} \rVert^2_2 + \epsilon) \lVert y-\hat{y} \rVert^2_2 \leq (\lVert \mathbf{g} \rVert^2_2 + \epsilon) \lVert y-\hat{y} \rVert^2_1\\ \leq {}&2 (\lVert \mathbf{g} \rVert^2_2 + \epsilon) \mathrm{KL}(\hat{y},y) = 2 (\lVert \mathbf{g} \rVert^2_2 + \epsilon) D_{\phi}(\hat{y},y). \end{aligned}
\end{equation}
By \cref{thm:A2}, \cref{assump: A2} holds with constant 
$$ C_{f,\phi}\leq \frac{L_{\fb}(A)}{\mu_{\phi}(A)} = \lVert \mathbf{g}\rVert_2^2 + \epsilon \to \lVert \mathbf{g}\rVert_2^2. $$
Since $\epsilon$ is arbitrary, we take the limit $\epsilon\to 0+$ to get a tighter upper bound. $\Box$
\end{example}

Choosing a mirror function $\phi$ that enforces feasibility allows the method to handle constrained optimization naturally.

\begin{example}[Max-margin model]\label{ex:max-margin}\rm
We consider the max-margin dual model from~\cite{Collins2008}:
$$ \min_{x \in \Delta^d} f(x)= b^{\top}x + \frac12 x^{\top} A x, $$
where $\Delta^d$ is the probability simplex, $b\in\mathbb R^d$, and $A$ is SPD. We use Shannon entropy $\phi(x)=\sum_{i=1}^d x_i\log x_i$ as the mirror function to keep iterates $y_k$ in $\Delta^d$. In the relatively convex case $\mu=0$, we have $L_{\fb}(A)=L_f(A)= 1$. Similar to~\eqref{eq:Pinsker}, $\mu_\phi(A)=1/\|A\|_2$, so 
$$ C_{f,\phi}=\|A\|_2=\lambda_{\max}(A), $$
which the power method can estimate. We extend Acc-MD to the case $\mu=0$ in \cref{alg: restart perturbed VOS}. $\Box$
\end{example}

In Examples \ref{ex: log-linear} and \ref{ex:max-margin}, $\fb$ is quadratic and admits a global metric $A$. For $\mathcal C^2$ mirror functions, the Hessians $\nabla^2\phi(x)$ and $\nabla^2\phi(y)$ usually differ. We use continuity to bound this difference.

We assume $\phi$ has a Lipschitz continuous Hessian in the ball $B(x^{\star}, R/2)$ with constant $L_{\nabla^2\phi}(R)$. Specifically, $\phi \in \mathcal C^{2, \theta}(B(x^{\star}, R/2))$ if the Hessian is H{\"o}lder continuous with parameter $\theta \in [0,1]$: 
$$ \|\nabla^2 \phi(x) - \nabla^2 \phi(y)\|_{*} \leq L_{\nabla^2 \phi} \|x - y\|^{\theta}, \quad \forall x, y \in V. $$

\begin{theorem}\label{th:C2}
Suppose $\phi$ is $1$-strongly convex and \cref{assump: A1} holds. If $\phi \in \mathcal C^{2, \theta}(B(x^{\star},R/2))$ for $0\le \theta \le 1$, then \cref{assump: A2} holds locally for $(x,\hat x,y,\hat y)\in B(x^{\star},R/2)$ with
$$ C_{f,\phi}\le (L-\mu)\bigl(1+L_{\nabla^2\phi}(R)R^{\theta}\bigr), $$
where $L_{\nabla^2\phi}$ is the H\"older constant of $\nabla^2\phi$ on $B(x^{\star},R/2)$. If $\phi$ is quadratic, then $L_{\nabla^2\phi}=0$ and $C_{f,\phi}\le L-\mu$.
\end{theorem}
\begin{proof}
Fix $(x,\hat x,y,\hat y)\in B(x^\star,R/2)$. By Taylor's theorem, there exist points
$\xi_y$ on the segment $[\hat y,y]$ and $\xi_x$ on $[\hat x,x]$ such that
\[
2 D_\phi(\hat y,y) = \|y-\hat y\|_{A}^2,
\qquad
\langle \nabla\phi(x)-\nabla\phi(\hat x),\, x-\hat x \rangle
= \|x-\hat x\|_{H_x}^2,
\]
where
\[
A := \nabla^2\phi(\xi_y),
\qquad
H_x := \nabla^2\phi(\xi_x).
\]
Since $\phi$ is $1$-strongly convex, we have $A \geq I$, and therefore
$\|v\| \le \|v\|_{A}$ for all $v$.
The H\"older continuity of $\nabla^2\phi$ on $B(x^\star, R/2)$ ensures $$\|H_x - A\| \le L_{\nabla^2\phi} \|\xi_x - \xi_y\|^\theta \le L_{\nabla^2\phi} R^\theta.$$ It follows that
$$ \|x-\hat x\|_{H_x}^2 \le \|x-\hat x\|_A^2 + \|H_x - A\| \|x-\hat x\|^2 \le (1+L_{\nabla^2\phi}R^\theta) \|x-\hat x\|_A^2. $$
By relative smoothness, $\langle \nabla\fb(x)-\nabla\fb(\hat x), x-\hat x \rangle \le (L-\mu) \|x-\hat x\|_{H_x}^2$. Combining these bounds implies $\fb$ is $L_{\fb}(A)$-smooth in the $\|\cdot\|_A$ norm with
$$ L_{\fb}(A) \le (L-\mu)(1+L_{\nabla^2\phi}R^\theta). $$
The result then follows from co-convexity and the Cauchy--Schwarz argument used in \cref{thm:A2}.
\end{proof}

In practice, the intermediate points $\xi_x$ and $\xi_y$ from the Taylor expansion are unknown. To operationalize \cref{th:C2}, one can use a practical adaptive estimate for the GCS constant:
\begin{equation}\label{eq:practicalC}
C_{f,\phi} \approx (L- \mu) \bigl(1 + L_{\nabla^2\phi}(R)\| x_k - y_k\|^\theta\bigr).
\end{equation}
Alternatively, using  $L$ as a global estimate for $C_{f,\phi}$, i.e., 
\begin{equation}\label{eq:CL}
C_{f,\phi} \approx L.
\end{equation}
provides a robust and conservative upper bound that ensures algorithmic stability by bypassing the need to compute sensitive local Hessian curvature parameters or H\"older constants, which is particularly advantageous during the transient phase when iterates are far from the optimum; see Numerical Example \ref{sec:quartic}.

A similar result can be obtained by assuming that $\fb\in \mathcal C^{2,\theta}(B(x^\star,R/2))$. The proof is slightly involved as $\fb$ may not be strongly convex. 

\begin{theorem} \label{thm:A2C2fb}
Suppose $\phi$ is $1$-strongly convex and \cref{assump: A1} holds. If $f-\mu\phi\in \mathcal C^{2,\theta}(B(x^\star,R/2))$, then \cref{assump: A2} holds locally for any $(x,\hat x,y,\hat y)\in B(x^\star,R/2)$ with
$$ C_{f,\phi}\le 3\bigl((L-\mu)+2L_{\nabla^2\fb}(R)R^\theta\bigr). $$
\end{theorem}
\begin{proof}
By Taylor's theorem, there exist points $\xi, \zeta$ on the line segment
$[\hat{x}, x]$ and a point $\eta$ on $[\hat{y}, y]$ such that
\[
\nabla \fb(x) - \nabla \fb(\hat{x}) = A_1 (x - \hat{x}), \
2D_{\fb}(x, \hat{x}) = \|x - \hat{x}\|_{A_2}^2, \
2D_{\fb}(y, \hat{y}) = \|y - \hat{y}\|_{A_3}^2,
\]
where $A_1 = \nabla^2 \fb(\xi)$, $A_2 = \nabla^2 \fb(\zeta)$, and
$A_3 = \nabla^2 \fb(\eta)$.

Setting $A = A_2 + \varepsilon I$ for $\varepsilon > 0$, the Cauchy--Schwarz inequality implies
$$ \dual{\nabla\fb(x) - \nabla\fb(\hat{x}), y - \hat{y}} \leq \|\nabla\fb(x) - \nabla\fb(\hat{x})\|_{A^{-1}} \|y - \hat{y}\|_A. $$
H\"older continuity on $B(x^\star, R/2)$ ensures $\|A - A_3\| \leq L_{\nabla^2\fb}R^\theta + \varepsilon$. Using relative smoothness and the strong convexity of $\phi$, we have $$\|y - \hat{y}\|_A^2 \leq 2((L - \mu) + L_{\nabla^2\fb}R^\theta + \varepsilon)D_\phi(\hat{y}, y).$$

Next, co-coercivity of $\fb$ in the $\|\cdot\|_{A^{-1}}$ norm gives $$\|\nabla\fb(x) - \nabla\fb(\hat{x})\|_{A^{-1}}^2 \leq 2(2 + \varepsilon^{-1}L_{\nabla^2\fb}R^\theta)D_{\fb}(x, \hat{x}).$$ Combining these two bounds and setting the regularization parameter $\varepsilon = L_{\nabla^2\fb}R^\theta$ yields the desired constant $C_{f,\phi}$.
\end{proof}

\begin{example}\rm
Let $f(x) = 2x \log x$ for $x \le 1$ and $x \log x + 2x - \log x - 2$ for $x > 1$, with $\phi(x) = x \log x$ on $\mathbb{R}_+$. Direct calculation shows $f, \phi \in \mathcal{C}^2(\mathbb{R}_+)$ and \cref{assump: A1} holds with $L=2$ and $\mu=1$. Then
$$ \fb(x) = \begin{cases} x \log x, & 0 < x \le 1, \\ 2x - \log x - 2, & x > 1. \end{cases} $$
The global GCS condition \eqref{eq: A2} fails. As $x \to 0^+$, we have $\nabla \fb(x) = \log x + 1 \to -\infty$, but the divergence remains finite:
$$ D_{\fb}(x,1) = \fb(x) - \fb(1) - \fb'(1)(x-1) \to 1. $$
However, the local version of \cref{assump: A2} holds near the minimizer $x^\star = 1/e$ for $R \le 1/e$. This follows from Theorem \ref{thm:A2C2fb} because $\phi$ is strongly convex and both functions are smooth on $B(x^\star, R)$. $\Box$
\end{example}

\paragraph{Comparison between \cref{assump: A2} and TSE}
In~\cite{Hanzely2021}, $\phi$ has triangle scaling exponent $\gamma$ (TSE-$\gamma$) if:
\begin{equation}\label{eq:TSE}
D_\phi\bigl((1-\theta)x+\theta y,\,(1-\theta)x+\theta \tilde y\bigr) \le \theta^\gamma D_\phi(y,\tilde y), \qquad \forall \theta\in[0,1].
\end{equation}
ABPG achieves $\mathcal O(1/k^\gamma)$ convergence. For the Shannon entropy, $\gamma=1$ and ABPG does not accelerate. 

The notion of {\it intrinsic} TSE~\cite{Hanzely2021} scales~\eqref{eq:TSE} by a constant $C(x,y,\tilde y)$. Convex, twice differentiable mirror functions satisfy this with $\gamma_{\textbf{in}}=2$, yielding the ABPG rate $$ f(x_{k+1})-f(x^\star)=\mathcal O(C_k/k^2), $$ where $C_k$ depends on the trajectory. It is similarly to our local version of \cref{assump: A2}; see Theorem~\ref{th:C2}. 

As GCS and TSE capture different geometric structures, neither generally implies the other.  This relationship is summarized in Fig. \ref{fig:diagram}.

\begin{figure}[htbp]
\begin{center}
\includegraphics[width=3.6in]{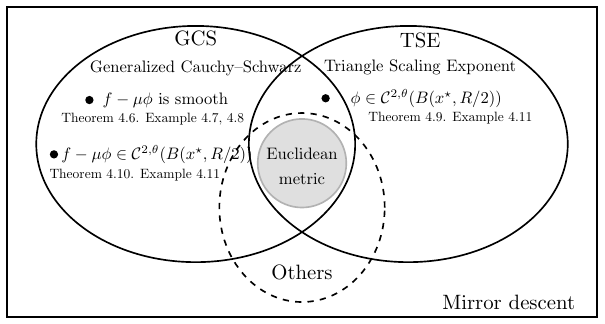}
\caption{Comparison of assumptions. For the entropy mirror, Acc-MD achieves acceleration under GCS, whereas ABPG fails due to the lack of a suitable TSE.}
\label{fig:diagram}
\end{center}
\end{figure}

\section{Extension to convex optimization}\label{sec: extension to convex opt}
For convex objectives ($\mu = 0$), we consider the following perturbed flow:
\begin{equation}\label{eq:perturbed AMD}
	\left\{
	\begin{aligned}
		x^{\prime} &= y - x, \\
		(\nabla \phi(y))^{\prime} &= - \epsilon^{-1}\nabla f(x) + \nabla \phi(x) - \nabla \phi(y),
	\end{aligned}
	\right.
\end{equation}
for a fixed perturbation level $\epsilon > 0$. In dual variables, this system becomes
$$ x' = y - x, \quad \eta' = -\epsilon^{-1} \nabla f(x) + \chi - \eta, $$
where $\chi = \nabla \phi(x)$ and $\eta = \nabla \phi(y)$. Due to the variable splitting, the perturbation does not change the stational points of the original system.

Consider the Lyapunov function
\begin{equation}\label{eq:perturbed lya}
   \mathcal E_{\epsilon}(x,\eta)
   := D_{f}(x,x^{\star}) + \epsilon\, D_{\phi^*}(\eta,\chi^{\star}),
\end{equation}
where $\epsilon>0$ is a fixed parameter. Taking derivatives with respect to $x$ and $\eta$ gives a modified energy dissipation identity.

\begin{lemma}[Perturbed strong Lyapunov property]\label{lem:pertrubed Lya property}
Let $\mathcal E_\epsilon$ be the Lyapunov function in~\eqref{eq:perturbed lya}, and let $\mathcal G$ be the vector field of the perturbed accelerated mirror descent flow~\eqref{eq:perturbed AMD}. Then:
\begin{equation}\label{eq:pertrubed Lya property}
-\nabla \mathcal{E}_\epsilon(x,\eta)\cdot \mathcal{G}(x,\eta)
= \mathcal{E}_\epsilon(x,\eta)
+ D_f(x^{\star},x)
+ \epsilon D_{\phi^*}(\chi,\eta)
- \epsilon D_{\phi^*}(\chi,\chi^{\star}).
\end{equation}
\end{lemma}

\begin{proof}
Direct calculation gives:
$$ \begin{aligned} - \nabla \mathcal{E}_\epsilon(x,\eta)\cdot \mathcal{G}(x,\eta) = {}& \langle \nabla f(x)-\nabla f(x^{\star}),\, x-x^{\star} \rangle + \epsilon\, \langle \nabla \phi^*(\eta)-\nabla \phi^*(\chi^{\star}),\, \eta-\chi \rangle . \end{aligned} $$
The first term is the symmetric Bregman divergence of $f$,
$$ \langle \nabla f(x)-\nabla f(x^{\star}),\, x-x^{\star} \rangle = D_f(x,x^{\star}) + D_f(x^{\star},x). $$
Using the three-point identity, the second term expands to
$$ \epsilon\, \langle \nabla \phi^*(\eta)-\nabla \phi^*(\chi^{\star}),\, \eta-\chi \rangle = \epsilon ( D_{\phi^*}(\eta,\chi^{\star}) + D_{\phi^*}(\chi,\eta) - D_{\phi^*}(\chi,\chi^{\star}) ). $$
Combining these with the definition~\eqref{eq:perturbed lya} of $\mathcal E_\epsilon$ yields~\eqref{eq:pertrubed Lya property}.
\end{proof}

The perturbed forward Acc-MD method is stated in \cref{alg:AccMD-perturbed-iteration}.

\begin{algorithm}
\caption{Perturbed Accelerated Mirror Descent Method}
\label{alg:AccMD-perturbed-iteration}
\linespread{1.125}\selectfont
\begin{algorithmic}[1]
 \STATE \textbf{Input:} $x_0, y_0 \in \mathbb{R}^n$, parameters $\epsilon > 0, C_{f,\phi} > 0.$ 

\STATE Set $\alpha =\sqrt{\epsilon/{C}_{f,\phi}}.$

\FOR{$k=0, 1, 2, \dots$}
\STATE $\displaystyle   x_{k+1} = \frac{1}{1+\alpha} (x_{k} + \alpha y_{k} ).$
\STATE $\displaystyle g_k = \alpha \nabla \phi(x_{k+1}) + \nabla \phi(y_k) - \frac{\alpha}{\epsilon} (2\nabla f(x_{k+1}) - \nabla f(x_k)).$

\STATE $\displaystyle y_{k+1}=\arg\min_{y \in \mathbb{R}^n} [ (1 + \alpha)\, \phi(y) - \langle g_k,\, y \rangle ]$.
\ENDFOR
\end{algorithmic}
\end{algorithm}

\paragraph{Convergence Analysis}
Define
\begin{align}
\mathcal{E}_{\epsilon}^\alpha(x,y) &:= D_{f}(x, x^{\star}) + \epsilon D_{\phi^*}(\eta, \chi^{\star}) +\alpha\langle \nabla f(x)-\nabla f(x^{\star}),\,y-x^{\star} \rangle, \label{eq:Eepsalpha}\\
\mathcal B^\alpha(x,\hat x,\hat y,y) &:= D_{f}(x,\hat x) + \epsilon D_{\phi}(\hat y,y) +\alpha\langle \nabla f(x)-\nabla f(\hat x),\,y-\hat y \rangle. \label{eq:Bepsalpha}
\end{align}
Using \cref{lem:pertrubed Lya property}, we obtain the following identity.

\begin{lemma}\label{lem:decay convex}
Let $(x_k,y_k)$ be the sequence from Algorithm~\ref{alg:AccMD-perturbed-iteration}. Then:
$$ \begin{aligned} \mathcal{E}_{\epsilon}^\alpha(x_{k+1},y_{k+1}) -\mathcal{E}_{\epsilon}^\alpha(x_k,y_k) ={}& -\alpha\,\mathcal E_{\epsilon}^{\alpha}(x_{k+1}, y_{k+1}) - \alpha \epsilon D_{\phi}(y_{k+1}, x_{k+1})\\ & + \alpha \epsilon D_{\phi}(x^{\star}, y_{k+1}) + \alpha \epsilon D_{\phi}(x^{\star},x_{k+1}) \\ & -\alpha\,\mathcal B^{-\alpha}(x^\star, x_{k+1},y_{k+1},x^\star) - \mathcal B^{\alpha}(x_k, x_{k+1},y_k,y_{k+1}). \end{aligned} $$
\end{lemma}

\begin{theorem}[Convergence of Perturbed Acc-MD]\label{thm:convergence rate of perturbed Acc-MD dual}
Assume $f$ is convex and \cref{assump: A2} holds with $\mu=0$ and $C_{f,\phi} > 0$. For any $\epsilon > 0$, let $(x_k, y_k)$ be the sequence generated by Algorithm \ref{alg:AccMD-perturbed-iteration} with $\alpha = \sqrt{\epsilon/C_{f,\phi}}$. If there exists $R > 0$ such that $D_{\phi}(x^\star, x_k) \le R/2$ and $D_{\phi}(x^\star, y_k) \le R/2$ for all $k \ge 0$, then:
\begin{equation}\label{eq: perturbed AOR-AMD decay}
\mathcal{E}^{\alpha}_{\epsilon}(x_{k}, y_{k}) \le \left( 1 + \sqrt{\epsilon/C_{f,\phi}} \right)^{-k} \mathcal{E}^{\alpha}_{\epsilon}(x_{0}, y_{0}) + \epsilon R.
\end{equation}
\end{theorem}
\begin{proof}
For $\alpha = \sqrt{\epsilon/C_{f,\phi}}$, Lemma \ref{lem:crossterm} ensures $\mathcal{B}^\alpha \ge 0$ and $\mathcal{B}^{-\alpha} \ge 0$. Lemma \ref{lem:decay convex} then implies:
$$
\mathcal{E}^\alpha_{\epsilon}(x_{k+1},y_{k+1}) \le \frac{1}{1+\alpha} \mathcal{E}^\alpha_{\epsilon}(x_k,y_k) + \frac{\alpha\epsilon R}{1+\alpha}.
$$
Iterating this inequality and summing the geometric series yields \eqref{eq: perturbed AOR-AMD decay}.
\end{proof}

Since the perturbation preserves the equilibrium $x^\star$, we employ a homotopy strategy that decreases $\epsilon$ to achieve the optimal $\mathcal{O}(C_{f,\phi}/k^2)$ convergence rate (\cref{alg: restart perturbed VOS}).

\begin{algorithm}[htbp]
\caption{Acc-MD with Homotopy Perturbation}
\linespread{1.125}\selectfont
\label{alg: restart perturbed VOS}
\begin{algorithmic}[1]
    \STATE \textbf{Input:} $x_0, y_0$, initial tolerance $\epsilon_0 > 0$, termination $\epsilon > 0$.
    \STATE \textbf{Set} $m_0 = C\epsilon_0^{-1/2}$.
    \FOR{$k=0, 1, 2, \dots$}
        \STATE $\epsilon_{k+1} = \epsilon_k/2, \quad m_{k+1} = \sqrt{2} m_k$.
        \STATE Run Algorithm \ref{alg:AccMD-perturbed-iteration} for $m_{k+1}$ iterations with parameter $\epsilon_{k+1}$ and initial state $(x_k, y_k)$ to obtain $(x_{k+1}, y_{k+1})$.
    \ENDFOR
\end{algorithmic}
\end{algorithm}

Define the Lyapunov function $\mathcal{E}(x, \eta, \epsilon) := D_f(x, x^{\star}) + \epsilon D_{\phi^*}(\eta, \chi^{\star})$. Assuming $D_{\phi}(x^\star, x_0) \le R^2$, we have the following global result. The proof is identical to  \cite[Theorem 8.4]{chen2025accelerated} and thus skipped here.

\begin{theorem}
Under \cref{assump: A2} with $\mu=0$, if $\mathcal E(x_0, \eta_0, \epsilon_0) \le (R^2+1)\epsilon_0$, then Algorithm \ref{alg: restart perturbed VOS} maintains $\mathcal E(x_k, \eta_k, \epsilon_k) \le (R^2+1)\epsilon_k$ for all $k \ge 0$. Furthermore, for $M_k = \sum_{i=1}^k m_i$ total iterations, there exists $c > 0$ such that:
$$
\mathcal E(x_k, \eta_k, \epsilon_k) \le \frac{c C_{f,\phi} (R^2+1)}{M_k^2}.
$$
This corresponds to an iteration complexity of $\mathcal{O}(\sqrt{C_{f,\phi}/\epsilon})$.
\end{theorem}

\section{Extension to Composite Optimization}\label{sec: extension to composite opt}
Consider the composite optimization problem:
\begin{equation}
\min_{x \in \mathbb{R}^n} \quad F(x) := f(x) + g(x),
\end{equation}
where $f$ satisfies \cref{assump: A1} and $g$ is convex but possibly non-smooth. This formulation encompasses constrained optimization by letting $g$ be the indicator function of a convex set $K$. We assume $g$ admits a tractable proximal operator with respect to the mirror function $\phi$. Crucially, \eqref{eq: A1} and \eqref{eq: A2} are required only for $f$. 

The composite accelerated mirror descent flow is defined as:
\begin{equation}\label{eq:VOSflow-composite}  
\left\{
\begin{aligned}
    x^{\prime} &= y - x, \\
    -\mu (\nabla \phi(y))^{\prime} + \mu(\nabla \phi(x)-\nabla \phi(y)) - \nabla f(x) &\in \partial g(y).
\end{aligned}
\right.
\end{equation}
In dual variables $\eta = \nabla \phi(y)$, this system simplifies to:
$$ x' = y - x, \qquad \eta' = -\mu^{-1}(\nabla \fb(x)+q(y)) - \eta, $$
where $q(y) \in \partial g(y)$ is a subgradient of $g$. We employ the Lyapunov function \eqref{eq:lya_func} used in the smooth case; the strong Lyapunov property follows by an analogous argument.

The backward Acc-MD scheme for composite optimization is summarized in Algorithm~\ref{alg:AccMD-composite}, and the forward version can be obtained similarly.

\begin{algorithm}
\caption{Accelerated Mirror Descent Method for Composite Optimization}
\label{alg:AccMD-composite}
\linespread{1.325}\selectfont
\begin{algorithmic}[1]
\STATE \textbf{Input:} $x_0, y_0 \in \mathbb{R}^n$, parameters $\mu>0$, $C_{f,\phi}>0$
\STATE \textbf{Set $\alpha =\sqrt{\mu/{C}_{f,\phi}}.$}
\FOR{$k=0, 1, 2, \dots$}

\STATE $y_{k+1} = \arg\min_{y \in \mathbb{R}^n} (1 + \alpha)\, \phi(y) + \frac{\alpha}{\mu} g(y) - \left\langle \alpha \nabla \phi(x_k) + \nabla \phi(y_k) - \frac{\alpha}{\mu} \nabla f(x_k),\, y \right\rangle.$

\STATE $\displaystyle x_{k+1} = \tfrac{1}{1+\alpha} [x_k + \alpha( 2y_{k+1} - y_k) ]$.
\ENDFOR
\end{algorithmic}
\end{algorithm}

As subgradients of $g$ are evaluated implicitly, the non-smooth term $g$ does not contribute to the cross-term error, and convergence follows.

\begin{theorem}[Convergence in the Composite Setting]\label{thm:convergence rate of Acc-MD dual comp}
Let $g$ be convex and $f$ be $\mu$-relatively convex with respect to $\phi$ ($\mu > 0$). If \cref{assump: A2} holds with constant $C_{f,\phi} > 0$, the sequence $(x_k, y_k)$ generated by Algorithm~\ref{alg:AccMD-composite} with $\alpha = \sqrt{\mu/C_{f,\phi}}$ satisfies:
$$ D_{\fb}(x_{k+1}, x^{\star}) + \mu D_{\phi^*}(\eta_{k+1}, \chi^{\star}) \leq C_0 \left( 1 + \sqrt{\mu/C_{f,\phi}} \right)^{-k}, \quad k \geq 1, $$
where $C_0$ is a constant depending on the initial conditions.
\end{theorem}

\begin{example}[LASSO problem]\label{ex: Lasso}\rm
Consider the LASSO problem
\begin{equation}\label{eq:lasso}
    \min_{x \in \mathbb{R}^d} F(x) := \frac12 \|Ax - b\|^2 + \lambda \|x\|_1, \quad \phi(x)=\frac12 x^\top D x,
\end{equation}
where $A \in \mathbb{R}^{n \times d}$ with $n<d$ and $D=\mathrm{diag}(A^\top A)$. Since $D$ is diagonal and positive definite, the Acc-MD subproblem (line~4 in Algorithm~\ref{alg:AccMD-composite}) admits a closed-form solution given by a generalized soft-thresholding operator. According to \cref{th:C2}, we may take $C_{f,\phi}=L=\rho\!\left(D^{-1/2}A^\top A D^{-1/2}\right),$
where $\rho(\cdot)$ denotes the spectral radius.
\end{example}

\section{Numerical Examples}\label{sec: numerical}

We evaluate the performance of Acc-MD on a variety of convex optimization problems, ensuring reproducibility by fixing all random seeds. We compare Acc-MD against several state-of-the-art first-order methods from the literature. Numerical experiment demonstrate that Acc-MD exhibits superior stability and competitive acceleration compared to existing variants.



\subsection{Quartic Objective}\label{sec:quartic}
We consider the quartic objective and mirror function from \cite[Section 2.1]{lu2018relatively}:
\begin{align*}
    f(x) = \frac{1}{4}\|Ex\|_2^4 + \frac{1}{4}\|Ax - b\|_4^4 + \frac{1}{2}\|Cx - d\|_2^2, \quad 
    \phi(x) = \frac{1}{4}\|x\|_2^4 + \frac{1}{2}\|x\|_2^2,
\end{align*}
where $C$ and $E$ are $n \times n$ positive definite matrices. The mirror map can be computed efficiently; see \cite{lu2018relatively} for details.

Let $\lambda_C$ and $\lambda_E$ denote the smallest eigenvalues of $C$ and $E$, respectively. Since $\|\nabla^2 f\|$ grows quadratically with $\|x\|_2$, $\nabla f$ is not globally Lipschitz. To apply NAG and AOR-HB, we assume $\|x\|_2 \leq R$ and estimate global parameters as:
\[
L_f(R)= (3\|E\|^4 + 3\|A\|^4) R^2 + 6\|A\|^3 \|b\|_2 R + 3\|A\|^2 \|b\|_2^2 + \|C\|^2, 
\qquad
\mu = \lambda_C^2.
\]
According to \cite{lu2018relatively}, \cref{assump: A1} holds with relative smoothness constant $L = L_f(1)$ and relative strong convexity constant $\mu = \min\left\{ \lambda_E^4/3,\ \lambda_C^2 \right\}.$ 

Since $\phi$ is smooth, we apply Theorem~\ref{th:C2} to estimate $C_{f, \phi}$. We denote by Acc-MD-forward/backward-1 using \eqref{eq:practicalC} $C_{f,\phi} \approx (L- \mu) \bigl(1 + L_{\nabla^2\phi}(R)\| x_k - y_k\|\bigr)$, and those using \eqref{eq:CL} $C_{f,\phi}\approx L$ as Acc-MD-forward/backward-2. 

We set $n=256$, $A \sim \frac{1}{\sqrt{n}} \mathcal{N}(\mathbf{0}, I_{n \times n})$, $C = I_n + C_0C_0^\top/n$ and $E = 2I_n + E_0E_0^\top/n$ with $C_0, E_0 \sim \mathcal{N}(\mathbf{0}, I_{n \times n})$, $b = 0$, and $d \sim \mathrm{Unif}(0,1)^n$. This setup yields $\mu = \mu_\phi = 1$. While $L$ remains fixed, $L_f$ grows with $R = \|d\|_2$, highlighting the benefit of relative smoothness.

We evaluate Acc-MD against several baselines: Nesterov’s accelerated mirror descent (NAMD), Nesterov’s accelerated gradient (NAG), accelerated over-relaxation heavy ball (AOR-HB), and standard mirror descent (MD). All methods use the stopping criterion $\|\nabla f(x_k)\|^2\le \mathrm{tol}\,\|\nabla f(x_0)\|^2$ with $\mathrm{tol}=10^{-12}$.

\begin{figure}[htp]
    \centering
    \begin{minipage}[t]{0.5\textwidth}
        \vspace{0pt} 
        Runtime results are presented in Figure~\ref{fig:quartic-exetime-adc}. Acc-MD variants converge faster and more stably than all baselines. Within the Acc-MD family, forward and backward variants exhibit nearly identical performance. 

        Notably, variants using $C_{f,\phi}$ from \eqref{eq:CL} outperform those using the practical estimate \eqref{eq:practicalC} in the initial stages. This is due to the large distance $\|x_k - y_k\|$ in early iterations, which renders the estimate in \eqref{eq:practicalC} loose. As iterates stabilize, both approaches achieve the same asymptotic accelerated linear rate.
    \end{minipage}
    \hfill
    \begin{minipage}[t]{0.45\textwidth}
        \vspace{0pt}
        \centering
        \includegraphics[width=1.0\linewidth]{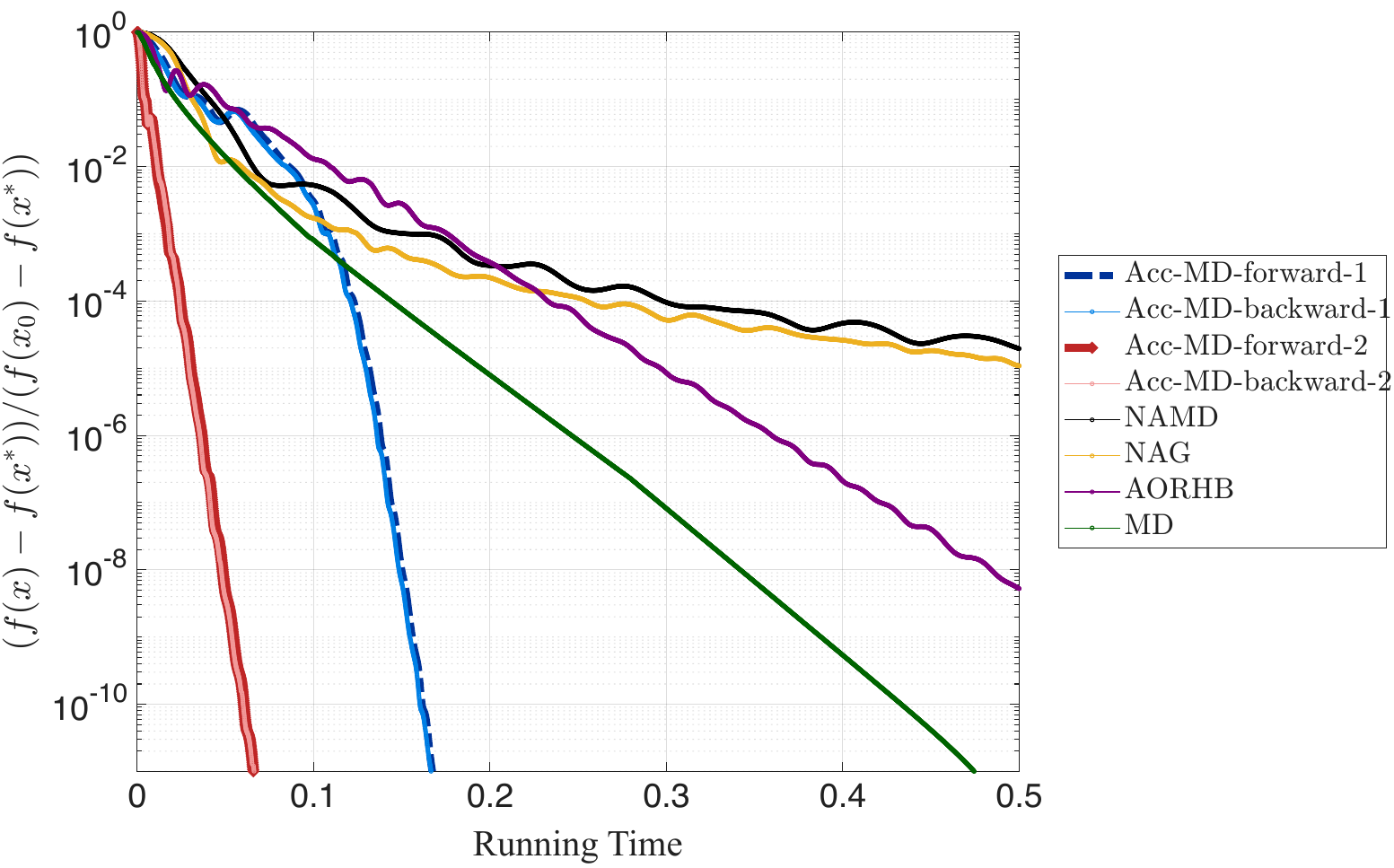}
        \caption{Log relative error vs. running time for the quartic problem.}
        \label{fig:quartic-exetime-adc}
    \end{minipage}
\end{figure}

\subsection{Entropic Mirror Descent}
Consider the log-linear model (Example~\ref{ex: log-linear}), where $\mu=1$ and $C_{f,\phi}=\|\mathbf{g}\|_2^2$. In this setting, Acc-MD achieves an accelerated linear rate of $(1+1/\|\mathbf{g}\|_2)^{-1}$. In contrast, existing entropic mirror descent methods attain, at best, a sublinear rate of $\mathcal{O}(1/k^2)$.

We compare Acc-MD with the Accelerated Bregman Proximal Gradient (ABPG) and Bregman Proximal Gradient (BPG) methods from \cite{Hanzely2021}, as shown in Figure~\ref{fig:entropicMDexample}. For this problem, the TSE exponent is $\gamma=1$; consequently, ABPG only guarantees a non-accelerated rate of $\mathcal{O}(1/k)$ and performs slightly slower than BPG, which is consistent with the findings in \cite[Section~6.1.1]{Hanzely2021}. Figure~\ref{fig:entropicMDexample} demonstrates that Acc-MD markedly outperforms both baselines, exhibiting fast and stable error decay.

\begin{figure}[h]
 \begin{minipage}{0.475\textwidth}
 \centering
  \includegraphics[width=0.8\linewidth]{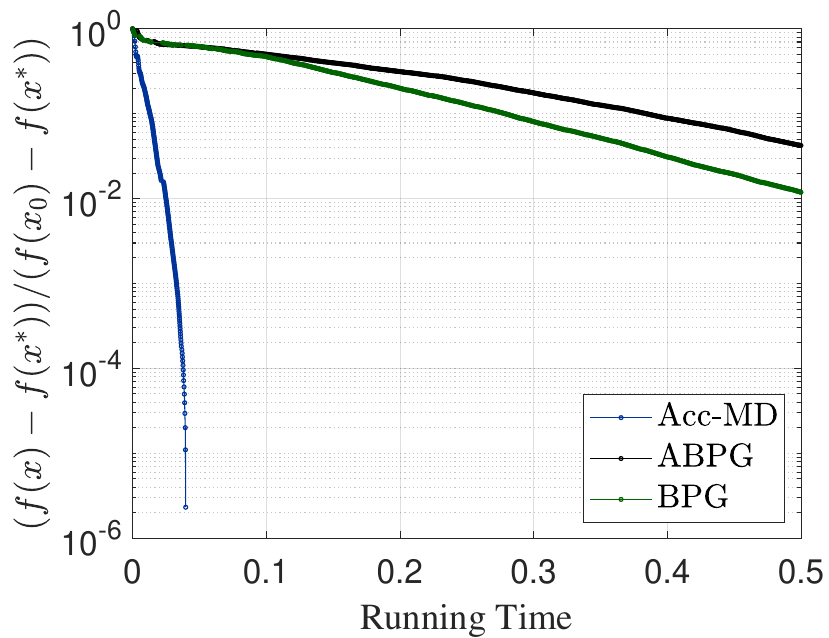}
        \caption{Log relative error vs. running time curve on the entropic MD problem.}
        \label{fig:entropicMDexample}
\end{minipage}
 \hfill
   \begin{minipage}{0.465\textwidth}
 \centering
        \includegraphics[width=0.8\linewidth]{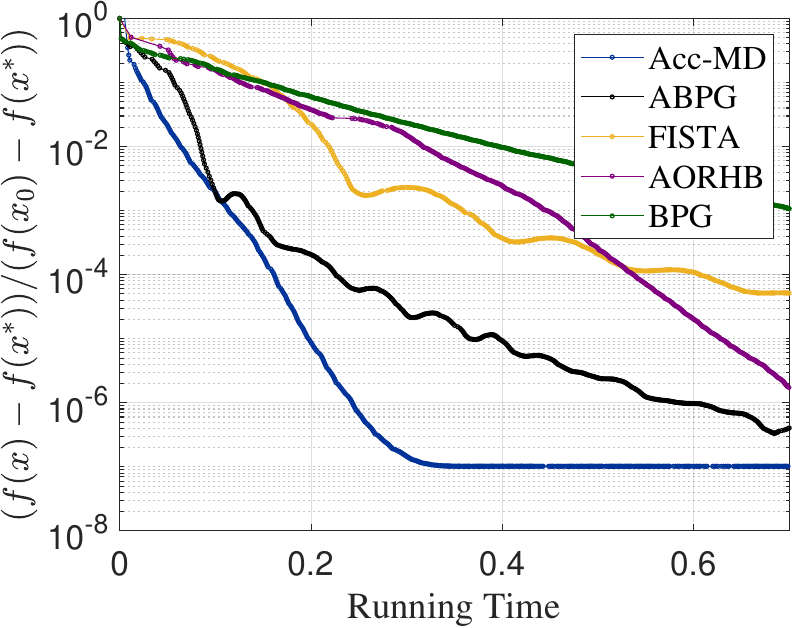}
    \caption{Relative error vs. running time curve for the LASSO problem.}
    \label{fig:leukemialasso}
\end{minipage}
\end{figure}


\subsection{LASSO Problem}
We evaluate the over-parameterized LASSO problem (Example~\ref{ex: Lasso}) using the Leukemia dataset, a standard benchmark in the LASSO literature \cite{sehic22a}. The dataset contains $7{,}129$ gene-expression features across $72$ samples for Leukemia type prediction. We set $\lambda = 0.05$; while larger regularization values induce stronger sparsity and accelerate convergence across all methods, we choose a smaller value to better accentuate performance differences.

Since the model is over-parameterized ($n < d$), the strong convexity parameter is $\mu = 0$. Consequently, we apply the perturbed variant of Acc-MD (Algorithm \ref{alg: restart perturbed VOS}). As illustrated in Figure \ref{fig:leukemialasso}, Acc-MD converges markedly faster than competing approaches and naturally exhibits an early-stopping effect.

\section{Conclusion}
In this paper, we introduced accelerated mirror descent (Acc-MD) methods that achieve improved convergence rates under relative smoothness and relative convexity settings. By leveraging relative strong convexity and a new generalized Cauchy-Schwarz (GCS) inequality, we established the first accelerated linear convergence rate for mirror descent methods. In the standard convex setting, we utilized a perturbation–homotopy argument to recover the optimal $\mathcal{O}(1/k^2)$ rate. Our proposed framework naturally extends to composite and constrained optimization problems while maintaining these accelerated guarantees.

We emphasize that \cref{assump: A2} is not intended to be universally verifiable or strictly stronger than existing assumptions. Relaxing this condition, or adaptively updating the constant $C_{f,\phi}$, represents an important direction for future research. Another promising path is the extension of this framework to stochastic settings, which is highly relevant for large-scale applications such as neural network training.

\bibliographystyle{siamplain}
\bibliography{references}
\end{document}